\documentclass[twoside]{amsart}

\numberwithin{equation}{section}
\usepackage{amsmath,amssymb,amsfonts,amsthm,latexsym}
\usepackage{graphicx,color,enumerate}
\usepackage[margin=1in]{geometry}

\newtheorem{theorem}{Theorem}[section]
\newtheorem{lemma}[theorem]{Lemma}
\newtheorem{proposition}[theorem]{Proposition}

\theoremstyle{definition}

\theoremstyle{remark}

\numberwithin{equation}{section}

\newcommand{\C}{\mathbb{C}}
\newcommand{\R}{\mathbb{R}}
\newcommand{\Z}{\mathbb{Z}}

\begin{document}

\title{A symmetric function approach to polynomial regression}

\author[H.-C.~Herbig]{Hans-Christian Herbig}
\address{Departamento de Matem\'{a}tica Aplicada,
Av. Athos da Silveira Ramos 149, Centro de Tecnologia - Bloco C, CEP: 21941-909 - Rio de Janeiro, Brazil}
\email{herbighc@gmail.com}

\author[D.~Herden]{Daniel Herden}
\address{Department of Mathematics, Baylor University,
Sid Richardson Building,
1410 S.4th Street,
Waco, TX 76706, USA}
\email{daniel\_herden@baylor.edu}

\author[C.~Seaton]{Christopher Seaton}
\address{Department of Mathematics and Statistics,
Rhodes College, 2000 N. Parkway, Memphis, TN 38112, USA}
\email{seatonc@rhodes.edu}

\thanks{C.S. was supported by an
AMS-Simons Research Enhancement Grant for PUI Faculty and a
Rhodes College Faculty Development Grant.}


\begin{abstract}
We give an explicit solution formula for the polynomial regression problem in terms of Schur polynomials and Vandermonde determinants. We thereby generalize the work of Chang, Deng, and Floater to the case of model functions of the form $\sum _{i=1}^{n} a_{i} x^{d_{i}}$ for some integer exponents $d_{1}  >d_{2}  >\dotsc  >d_{n} \geq 0$ and phrase the results using Schur polynomials. Even though the solution circumvents the well-known problems with the forward stability of the normal equation, it is only of practical value if $n$ is small because the number of terms in the formula grows rapidly with the number $m$ of data points. The formula can be evaluated essentially without rounding.
\end{abstract}

\maketitle

% xxxxxxxxxxxxxxxxxxxxxxxxxxxxxxxxxxxxxxxxxxxxxxxxxxxxxxxxxxxxxxxxxxxxxxxxx
% xxxxxxxxxxxxxxxxxxxxxxxxxxxxxxxxxxxxxxxxxxxxxxxxxxxxxxxxxxxxxxxxxxxxxxxxx
% xxxxxxxxxxxxxxxxxxxxxxxxxxxxxxxxxxxxxxxxxxxxxxxxxxxxxxxxxxxxxxxxxxxxxxxxx

\section{Introduction}
\label{sec:Intro}

Linear regression goes back to ideas of A.M. Legendre \cite{Legendre} and C.F. Gauß \cite{Gauss}, while
the first work on nonlinear regression is from 1815 and is due to J.D. Gergonne \cite{Gergonne}. Nowadays regression is implemented in most of the mathematical software as a numerical routine and is widely used in statistics and science.  The basic idea is to fit a model function, depending linearly on some parameters, to a set of data points in such a way that the sum of squares of the errors of the approximation is minimal.

In this paper we discuss the univariate case where the model function is a polynomial of the following form.
Let $\boldsymbol d=(d_{1} ,d_{2} ,\dotsc ,d_{n})$ be an $n$-tuple of integers such that $d_{1} > d_{2} > \dotsc > d_{n} \ge 0$. % and put $D:=\sum_{i=1}^n d_i$.
By a \emph{polynomial of type $\boldsymbol d$} we mean a univariate function of the form $f(x) =\sum _{i=1}^{n} a_{i} x^{d_{i}}$ with $a_{i}\in \C$. Let us assume that we have data points $\boldsymbol x=(x_{1} ,x_{2} ,\dotsc ,x_{m})^T$ and $\boldsymbol y=(y_{1} ,y_{2} ,\dotsc ,y_{m})^T\in \C^m$.
%We do not require $x_{1} ,x_{2} ,\dotsc ,x_{m}$ to be distinct.
The interpolation problem $f(x_{i}) =y_{i}$ for $i=1,\dotsc ,m$ is equivalent to the linear system
\begin{align}\label{eq:k-interp}
 \begin{pmatrix}
x_{1}^{d_{1}} & x_{1}^{d_{2}} & \dotsc & x_{1}^{d_{n}}\\
x_{2}^{d_{1}} & x_{2}^{d_{2}} & \dotsc & x_{2}^{d_{n}}\\
\vdots  & & \ddots & \vdots \\
x_{m}^{d_{1}} & x_{m}^{d_{2}} & \dotsc  & x_{m}^{d_{n}}
\end{pmatrix}\begin{pmatrix}
a_{1}\\
a_{2}\\
\vdots \\
a_{n}
\end{pmatrix} =\begin{pmatrix}
y_{1}\\
y_{2}\\
\vdots \\
y_{m}
\end{pmatrix} .
\end{align}
Usually, this system is overdetermined $(m>n)$ and inconsistent.  Let us denote the coefficient matrix of this system by $A\in \C^{m\times n}$ and use the shorthand notation $A\boldsymbol a=\boldsymbol y$ for \eqref{eq:k-interp}. Then the associated \emph{normal equation}
\begin{align}\label{eq:k-normal}
A^*A\boldsymbol a=A^*\boldsymbol y,
\end{align}
is consistent, where $A^*=\overline{A}^T$ denotes the Hermitian transpose of $A$. In fact,  $\boldsymbol a=(a_{1} ,a_{2} ,\dotsc ,a_{n})^T$ is a solution to \eqref{eq:k-normal} if and only if $f(x) =\sum _{i=1}^{n} a_{i} x^{d_{i}}$ minimizes the distance $\sqrt{\sum_{i=1}^m|f(x_i)-y_i|^2}=|\!|A\boldsymbol a-\boldsymbol y|\!|_2=d(A\boldsymbol a,\boldsymbol y)$. Here we use the Hermitian inner product
%$\langle \boldsymbol{u},\boldsymbol{v}\rangle = \boldsymbol{v}^\ast\boldsymbol{u}$
$\langle \boldsymbol{u} | \boldsymbol{v}\rangle = \boldsymbol{u}^\ast\boldsymbol{v}$
on $\C^m$.\footnote{We will use the convention that an inner product is antilinear in the first argument. It has its origin in Dirac's bra-ket of quantum mechanics and is widely used in the mathematical literature (see, e.g. \cite{ReedSimon}). A reader only acquainted with the opposite convention is advised to read the formulas backwards.} For $\boldsymbol d=(n-1,n-2,\dotsc ,1,0)$ this is known as the problem of \emph{polynomial regression of degree $n-1$}. The resulting minimal distance $d(A\boldsymbol a,\boldsymbol y)$ is given by $\sqrt{ |\!|\boldsymbol y|\!|^2-\langle A^*\boldsymbol y| \boldsymbol a\rangle}$. Note that the coefficient matrix of \eqref{eq:k-normal},
\begin{align}\label{eq:Hankel}
A^*A=\left(\sum\limits _{l=1}^{m}\overline{x}_{l}^{d_{i}} x_{l}^{d_{j}}\right)_{i,j=1,\dots,n},
\end{align}
is the \emph{Hankel matrix} of the sequence of power sums if $\boldsymbol x\in\R^m$ and $\boldsymbol d=(n-1, n-2, \dotsc ,1,0)$.

In this paper we exploit the fact that the normal equation \eqref{eq:k-normal} is equivariant with respect to simultaneous permutation of the $x_i$s and the $y_i$s. It has a unique solution if and only if %the columns of $A$ are linearly independent,
$A$ is injective
and we will use symmetric polynomials to provide concrete expressions for the solution (see Theorem \ref{thm:sol} below) in the case of injective $A$. If $\boldsymbol d=(n-1,n-2,\dotsc ,1,0)$ the solution is unique if and only if there are at least $n$ distinct values among the data points $x_i$. In fact, we will construct for any injective $A$ a matrix $B\in \mathbb{C}^{n\times \binom{m}{n-1}}$ such that the solution operator $A^{+}=(A^{*}A)^{-1}A^*$ to the least squares problem ${A\boldsymbol a=\boldsymbol y}$ (also known as the \emph{Moore-Penrose pseudoinverse}) can be written as
\begin{align}\label{eq:pseudo}
A^{+} =BB^{*} A^{*} =B( AB)^{*}.
\end{align}
The projection to the image of $A$ is then $P=AB( AB)^{*}$. Writing the complementary projection as $P^{\perp } =\boldsymbol{1} -P=\boldsymbol{1} -AB( AB)^{*}$ the minimal distance can be expressed as $d( A\boldsymbol{a} ,\boldsymbol{y}) =\sqrt{\langle \boldsymbol{y} | P^{\perp }\boldsymbol{y} \rangle }$.
The formula for $\boldsymbol d=(n-1,n-2,\dotsc ,1,0)$ was already elaborated previously by Chang, Deng, and Floater in \cite{ChinesePaper} for weighted least squares without making use of symmetric functions. The weighted case will be discussed in Section \ref{sec:weighted}.

At this stage it is already apparent that the approach has practical limitations. The problem is that the number of columns of $B$ is $\binom{m}{n-1}$, which, in the typical situation when $m\gg n$, is only reasonable for roughly the values $n=2,3,4,5$. The number of floating point operations (FLOPs) for the naive matrix multiplication $BB^{*} A^{*}$ is then $\mathcal{O}\left(n^2\binom{m}{n-1}\right)$. For $m\gg n\gg 0$ this number of operations can be estimated by $\mathcal{O}\left(\frac{n^{2}}{e}\left(\frac{em}{n}\right)^{n-1}\right)$. Similar estimates appear when counting the floating point operations for evaluating $B$ (see Section \ref{sec:Eff}).

From the point of view of stability, however, the situation is much better. The least squares problem has a sensitivity that is measured by $\kappa _{2}( A)^{2}$ if the minimal distance $d(A\boldsymbol a,\boldsymbol y)$ of the least squares approximation is large and $\kappa _{2}( A)$ if it is small
(see \cite[Theorem~20.1]{Higham}). Here $\kappa _{2}( A)=|\!|A|\!|_2|\!|A^+|\!|_2$ denotes the \emph{condition number} of the matrix $A$. %\footnote{It can be expressed, for example, as the ratio of the biggest singular value of $A$ over the smallest singular value of $A$.}
In particular, the forward error for solving the normal equation directly is about $\kappa _{2}( A)^{2}$
(see \cite[Subsection 20.4]{Higham}).
Therefore, from the point of view of forward stability, the normal equation is considered problematic, in particular in situations where the conditioning $\kappa _{2}( A)$ can be large. It is known that for polynomial regression of degree $n-1$, where the matrices $A$ are of Vandermonde type this can be the case (see, for example, \cite{Pan}).
Our explicit solution does not have this defect. Problems with the stability can only occur in the polynomial evaluations and matrix multiplications involved in the solution formula. But in principle those calculations can be done essentially without rounding; division only occurs in the last step. When using floating point arithmetic the operations in the numerator and denominator of our formula concern merely addition and matrix multiplication,
which are not as problematic from the point of view of stability. However, if $m$ is large the number of these operations is huge.

\section*{Acknowledgments}
C.S. was supported by an AMS-Simons Research Enhancement Grant for PUI Faculty and a Rhodes College
Faculty Development Grant.

% xxxxxxxxxxxxxxxxxxxxxxxxxxxxxxxxxxxxxxxxxxxxxxxxxxxxxxxxxxxxxxxxxxxxxxxxx

\section{The solution formula}
\label{sec:Formula}

We now recall the definition of Schur polynomials; see \cite[Section~I.3]{Macdonald}, \cite[Section~4.4]{Sagan}, or \cite[Section~7.10]{Stanley}. Any decreasing sequence $\boldsymbol \lambda=(\lambda_1,\dots,\lambda_{n})$ of integers $\lambda_1\ge \lambda_2\ge\ldots\ge \lambda_{n}\ge 0$ represents a partition of $N=\sum_{i=1}^n \lambda_i$ into at most $n$ parts. We denote by $\boldsymbol\delta(n)$ the partition $(n-1,n-2,\dots,1,0)$ of $N={n \choose 2}$. For any such $\boldsymbol \lambda$ we define the \emph{alternating polynomial}
\begin{align*}
a_{\boldsymbol\lambda}(\boldsymbol z):=|(z^{\lambda_j}_i)_{i,j=1,\dots, n}|
\end{align*}
for $\boldsymbol z=(z_1,\dots,z_{n}) \in \C^n$. Note that $a_{\boldsymbol\delta(n)}(\boldsymbol z)=\prod_{1\le i<j\le n}(z_{i}-z_{j})=:V(\boldsymbol z)$ is the \emph{Vandermonde determinant}.
The \emph{Schur polynomial} associated to $\boldsymbol \lambda=(\lambda_1,\dots,\lambda_{n})$ is
\begin{align*}
s_{\boldsymbol\lambda}(\boldsymbol z):=\frac{a_{\boldsymbol\lambda+\boldsymbol\delta(n)}(\boldsymbol z)}{a_{\boldsymbol\delta(n)}(\boldsymbol z)}=\frac{|( z_i^{\lambda_j+n-j}  )_{i,j=1,\dots,n}|}{V(\boldsymbol z)},
\end{align*}
where the partition $\boldsymbol\lambda+\boldsymbol\delta(n)$ results from the componentwise addition of $\boldsymbol\lambda$ and $\boldsymbol \delta(n)$.
This does indeed define a polynomial as every alternating polynomial is divisible by the Vandermonde determinant. In fact,  $s_{\boldsymbol \lambda}$ is a polynomial with non-negative integer coefficients that can be determined by counting out semistandard Young tableaux of shape $\boldsymbol\lambda$;
see \cite[Section~4.4]{Sagan} or \cite[Section~7.10]{Stanley}.

Given $\boldsymbol x:=(x_1,\dots,x_m)^T\in \C^m$ and $\boldsymbol k=(k_1,k_2,\dots,k_n)$  with $1\le k_1<k_2<\ldots<k_n\le m$, we write
$\boldsymbol x_{\boldsymbol k}:=(x_{k_1},\dots,x_{k_n})^T\in \C^n$. We can view $\boldsymbol k$ as an element of ${[m]\choose n}$, the collection of all $n$-element subsets of $[m]=\{1,2,\ldots,m\}$. We note the following immediate result.

\begin{proposition} \label{prop:DetPartitions}
For all $m,n \in \Z_{>0}$, partitions $\boldsymbol \lambda=(\lambda_1,\dots,\lambda_{n})$ and $\boldsymbol \nu=(\nu_1,\dots,\nu_{n})$ and $x_1,x_2,\ldots,x_m \in \C$, we have
\begin{align*}\left|\left(\sum\limits_{k=1}^m\overline x_k^{\lambda_i+n-i}  x_k^{\nu_j+n-j}\right)_{i,j=1,\dots,n}\right|
=\sum\limits_{\boldsymbol k\in {[m]\choose n}}
\overline{s_{\boldsymbol \lambda}(\boldsymbol x_{\boldsymbol k})}
s_{\boldsymbol \nu}(\boldsymbol x_{\boldsymbol k})|V(\boldsymbol x_{\boldsymbol k})|^2.
\end{align*}
\end{proposition}

\begin{proof}
By the Cauchy-Binet formula we have
\begin{align*}
\left|\left(\sum\limits_{k=1}^m\overline x_k^{\lambda_i+n-i}  x_k^{\nu_j+n-j}\right)_{i,j=1,\dots,n}\right|
&=\left|(\overline x_k^{\lambda_i+n-i})^T_{k\in [m], i\in [n]}(x_k^{\nu_j+n-j})_{k\in [m], j\in [n]}
\right|\\
&=\sum\limits_{1\le k_1<\ldots<k_n\le m}\left|(\overline x_{k_i}^{\lambda_j+n-j})^T_{i,j=1,\dots,n}\right|\left|(x_{k_i}^{\nu_j+n-j})_{i,j=1,\dots,n}
\right|\\
&=\sum\limits_{1\le k_1<\ldots<k_n\le m}
s_{\boldsymbol \lambda}(\overline{ \boldsymbol x}_{\boldsymbol k})V(\overline{ \boldsymbol x}_{\boldsymbol k})
s_{\boldsymbol \nu}(\boldsymbol x_{\boldsymbol k})V(\boldsymbol x_{\boldsymbol k})\\
&%=\sum\limits_{1\le k_1<\ldots<k_n\le m}
%=\sum\limits_{\boldsymbol k\in {[m]\choose n}}
%\overline{s_{\boldsymbol \lambda}(\boldsymbol x_{\boldsymbol k})}
%s_{\boldsymbol \nu}(\boldsymbol x_{\boldsymbol k})\overline{V(\boldsymbol x_{\boldsymbol k})}V(\boldsymbol x_{\boldsymbol k})
=\sum\limits_{\boldsymbol k\in {[m]\choose n}}
\overline{s_{\boldsymbol \lambda}(\boldsymbol x_{\boldsymbol k})}
s_{\boldsymbol \nu}(\boldsymbol x_{\boldsymbol k})|V(\boldsymbol x_{\boldsymbol k})|^2. \qedhere
\end{align*}
\end{proof}

We associate to $\boldsymbol d=(d_{1} ,\dotsc ,d_{n})$ with $D=\sum_{i=1}^n d_i$ a partition $\boldsymbol \lambda=(\lambda_1,\dots,\lambda_{n})$ by $\boldsymbol \lambda = \boldsymbol d- \boldsymbol \delta(n)$, i.e., $\lambda_k:=d_k+k-n$ for $k=1,\dots, n$. Note that $\boldsymbol\lambda$ is a partition of $D-{n \choose 2}$. Moreover, let $\boldsymbol d\langle i\rangle=(d_1,\dots,d_{i-1},d_{i+1},\dots,d_n)$ and put $\boldsymbol \lambda[i]=\boldsymbol d\langle i\rangle -\boldsymbol \delta(n-1)$. We see that $\boldsymbol \lambda[i]$ is a partition of
$D-d_i-{n-1 \choose 2}$.
We apply Proposition \ref{prop:DetPartitions} to obtain the following result.

\begin{theorem} \label {thm:sol}\label{thm:schurformula} Let $x_{1} ,\dotsc ,x_{m}\in \C$ with $m\ge n$ such that the system \eqref{eq:k-normal} has a unique solution. Then, using the notation above, this unique solution $\boldsymbol a=(a_1,\dots,a_n)^T$ as a function of $\boldsymbol d,\boldsymbol x$ and $\boldsymbol y$ is given by
\begin{align}
\label{eq:aiSolGen}
a_i=\frac{\sum\limits_{j=1}^n (-1)^{i+j} \left( \sum\limits_{\boldsymbol l\in {[m]\choose {n-1}}}s_{\boldsymbol \lambda[i]}(\boldsymbol x_{\boldsymbol l})\overline{s_{\boldsymbol \lambda[j]}( \boldsymbol x_{\boldsymbol l})}
|V(\boldsymbol x_{\boldsymbol l})|^2 \right)\left(\sum\limits_{k=1}^m \overline x_k^{d_j}  y_k \right)
}{\sum\limits_{\boldsymbol k\in {[m]\choose n}}
|s_{\boldsymbol \lambda}(\boldsymbol x_{\boldsymbol k})V(\boldsymbol x_{\boldsymbol k})|^2}
\end{align}
for $i=1,\dots, n$. In particular, if there are at least $n$ distinct positive real values among $x_{1} ,\dotsc ,x_{m}$, then \eqref{eq:k-normal} has a unique solution.
\end{theorem}

\begin{proof}
The system \eqref{eq:k-normal} has a unique solution if and only if the matrix $A^*A$ is invertible, and we have $\boldsymbol a = (A^*A)^{-1}A^* \boldsymbol y$. Applying the cofactor formula for the inverse of $A^*A$ gives
\begin{align*}
a_i=\frac1{|A^*A|} \sum\limits_{j=1}^n (-1)^{i+j} M_{ji} (A^* \boldsymbol y)_j,
\end{align*}
where $M_{ji}$ denotes the determinant of the submatrix of $A^*A$ formed by deleting the $j$th row and $i$th column of $A^*A$ and
$(A^* \boldsymbol y)_j$ denotes the $j$th component of the vector $A^* \boldsymbol y$, which is given by
$(A^* \boldsymbol y)_j =\sum\limits_{k=1}^m \overline x_k^{d_j}  y_k$. Moreover, combining \eqref{eq:Hankel} and Proposition \ref{prop:DetPartitions} yields
\begin{align*}
|A^*A| &=
\left|\left(\sum\limits_{k=1}^m\overline x_k^{\lambda_i+n-i}  x_k^{\lambda_j+n-j}\right)_{i,j=1,\dots,n}\right| \\
&=\sum\limits_{\boldsymbol k\in {[m]\choose n}}
\overline{s_{\boldsymbol \lambda}(\boldsymbol x_{\boldsymbol k})}
s_{\boldsymbol \lambda}(\boldsymbol x_{\boldsymbol k})|V(\boldsymbol x_{\boldsymbol k})|^2
= \sum\limits_{\boldsymbol k\in {[m]\choose n}}
|s_{\boldsymbol \lambda}(\boldsymbol x_{\boldsymbol k})V(\boldsymbol x_{\boldsymbol k})|^2.
\end{align*}
A similar argument demonstrates that $M_{ji} = \sum\limits_{\boldsymbol l\in {[m]\choose {n-1}}}s_{\boldsymbol \lambda[i]}(\boldsymbol x_{\boldsymbol l})\overline{s_{\boldsymbol \lambda[j]}( \boldsymbol x_{\boldsymbol l})}
|V(\boldsymbol x_{\boldsymbol l})|^2$.

Finally, recall that $s_{\boldsymbol \lambda}$ is a polynomial with non-negative integer coefficients. Thus, $V(\boldsymbol x_{\boldsymbol k})$ is nonzero for any vector $\boldsymbol x_{\boldsymbol k}$ of distinct complex numbers, and $s_{\boldsymbol \lambda}(\boldsymbol x_{\boldsymbol k})$ is positive for any vector $\boldsymbol x_{\boldsymbol k}$ of positive real numbers.
\end{proof}

%This amounts to saying that the entry of $(A^*A)^{-1}$ at the $i$th row and $j$th column is
%\begin{align*}
%(-1)^{i+j}{ \sum\limits_{\boldsymbol l\in {[m]\choose {n-1}}}s_{\boldsymbol \lambda[i]}(\boldsymbol x_{\boldsymbol l})s_{\boldsymbol \lambda[j]}(\overline{ \boldsymbol x}_{\boldsymbol l})
%|V(\boldsymbol x_{\boldsymbol l})|^2
%\over \sum\limits_{\boldsymbol k\in {[m]\choose n}}
%s_{\boldsymbol \lambda}(\boldsymbol x_{\boldsymbol k})
%s_{\boldsymbol \lambda}(\overline{ \boldsymbol x}_{\boldsymbol k})
%|V(\boldsymbol x_{\boldsymbol k})|^2}=
%(-1)^{i+j}\frac{\sum\limits _{\boldsymbol{l} \in \binom{[m]}{n-1}} a_{\boldsymbol{\lambda } [i]} (\boldsymbol{x}_{\boldsymbol{l}} )a_{\boldsymbol{\lambda } [j]} (\overline{\boldsymbol{x}}_{\boldsymbol{l}} )}{\sum\limits _{\boldsymbol{k} \in \binom{[m]}{n}} |a_{{\boldsymbol \lambda }} ({x}_{{k}} )|^{2}}.
%\end{align*}

We can reformulate Theorem \ref{thm:schurformula} also in terms of the solution operator $A^{+}=BB^*A^*$ of the least squares problem, i.e., the Moore-Penrose pseudoinverse. The matrix $B\in \mathbb{C}^{n\times \binom{m}{n-1}}$ is given by the formula
\begin{align}\label{eq:B}
B_{i\boldsymbol{l}} :=\frac{( -1)^{i} a_{\boldsymbol{\lambda} [i]} (\boldsymbol{x}_{\boldsymbol{l}} )}{\sqrt{\sum\limits _{\boldsymbol{k} \in \binom{[m]}{n}} |a_{\boldsymbol{\lambda}} (\boldsymbol{x}_{\boldsymbol{k}} )|^{2}}}=\frac{( -1)^{i} s_{\boldsymbol{\lambda} [i]} (\boldsymbol{x}_{\boldsymbol{l}} )V(\boldsymbol{x}_{\boldsymbol{l}} )}{\sqrt{\sum\limits _{\boldsymbol{k} \in \binom{[m]}{n}} |s_{\boldsymbol{\lambda}} (\boldsymbol{x}_{\boldsymbol{k}} )V(\boldsymbol{x}_{\boldsymbol{k}} )|^{2}}},
\end{align}
where we understand the set $\binom{[ m]}{n-1}$ to be endowed with some fixed total order for indexing the columns of the matrix $B$.

% xxxxxxxxxxxxxxxxxxxxxxxxxxxxxxxxxxxxxxxxxxxxxxxxxxxxxxxxxxxxxxxxxxxxxxxxx

\section{Weighted regression}
\label{sec:weighted}

What has been said in Section \ref{sec:Formula} can be easily adapted to regression with weighted least squares (see \cite{ChinesePaper}). Let us first investigate what happens to the normal equation if we work with a non-standard inner product.

\begin{lemma}
Let $ W\in \operatorname{GL}_{m}(\mathbb{C})$ be an invertible matrix and define the inner product \ $ \langle \boldsymbol{u} |\boldsymbol{v} \rangle^\prime :=\boldsymbol{u}^{*} W^{*} W\boldsymbol{v}= \langle W\boldsymbol{u} |W\boldsymbol{v} \rangle$ on $ \mathbb{C}^{m}$.
Let $ A^{\circledast }$ denote the adjoint of $ A\in \mathbb{C}^{m\times n}=\operatorname{Hom}\left(\mathbb{C}^{n} ,\mathbb{C}^{m}\right)$ with respect to $ \langle \ |\ \rangle^\prime $, i.e., $ \langle A^{\circledast }\boldsymbol{u} |\boldsymbol{x} \rangle =\langle \boldsymbol{u} |A\boldsymbol{x} \rangle^\prime$ for all $ \boldsymbol{u} \in \mathbb{C}^{m} ,\boldsymbol{x} \in \mathbb{C}^{n}$ . Then the normal equation $ A^{\circledast } A\boldsymbol{a} =A^{\circledast }\boldsymbol{ y}$ for $ \boldsymbol{a} \in \mathbb{C}^{n}, \boldsymbol{y} \in \mathbb{C}^{m}$ can be written as $ ( WA)^{\ast } \ WA\ \boldsymbol{a} =( WA)^{\ast } W\ \boldsymbol{y}$.
\end{lemma}

\begin{proof}
Let $ G =W^{*} W$ be the Gram matrix so that $ \langle \boldsymbol{u} |\boldsymbol{v} \rangle^\prime =\boldsymbol{u}^{*} G\boldsymbol{v}$. We have $ \boldsymbol{u}^{*}(A^{\circledast })^{*}\boldsymbol{x} = (A^{\circledast }\boldsymbol{u})^{*}\boldsymbol{x} = \boldsymbol{u}^{*} GA \boldsymbol{x}$ for all $ \boldsymbol{u} \in \mathbb{C}^{m} ,\boldsymbol{x} \in \mathbb{C}^{n}$ so that $ (A^{\circledast })^{*} =GA$ and $A^{\circledast } =A^{*}G^{*}$. %This in turn can be written as $ A^{\circledast } =( G A)^{*}$ and
Hence $ A^{\circledast } A\boldsymbol{a} =A^{*} G^{*} A\boldsymbol{a} =( WA)^{\ast } WA \  \boldsymbol{a}$ \ and $ A^{\circledast }\boldsymbol{ y}= A^{*} G^{*}\boldsymbol{y} =( WA)^{\ast } W \boldsymbol{y}$.
\end{proof}

We now assume that $ W=\operatorname{diag}( w_{1} ,\dotsc ,w_{m})$ for some weights $ w_{1} ,\dotsc ,w_{m} \in \mathbb{C}^{\times } =\C -\{0\}$.
Instead of minimizing $ $$ \langle A\boldsymbol{a} -\boldsymbol{y} |A\boldsymbol{a} -\boldsymbol{y} \rangle $ we are now going to minimize $\langle A\boldsymbol{a} -\boldsymbol{y} |A\boldsymbol{a} -\boldsymbol{y} \rangle^\prime =\langle WA\boldsymbol{a} -W\boldsymbol{y} |WA\boldsymbol{a} -W\boldsymbol{y} \rangle$, that is, to solve $ ( WA)^{\ast } WA\ \boldsymbol{a} =( WA)^{\ast } W\boldsymbol{y}$. We introduce the shorthand notation
$$ | w_{\boldsymbol{k}}| ^{2} :=\prod _{j=1}^{n}| w_{k_{j}}| ^{2},$$
where $ \boldsymbol{k} \in \binom{[m]}{n}$ is interpreted as $\boldsymbol k=(k_1,k_2,\dots,k_n)$  with $ 1\leq k_{1} < k_{2} < \dotsc < k_{n} \leq m$. A little modification of the argument of Proposition \ref{prop:DetPartitions} yields the formula
$$ \left| \left(\sum\limits _{k=1}^{m} |w_{k} |^{2}\overline{x}_{k}^{\lambda _{i} +n-i} x_{k}^{\nu _{j} +n-j}\right)_{i,j=1,\dotsc ,n}\right| =\sum\limits _{\boldsymbol{k} \in \binom{[m]}{n} \ }| w_{\boldsymbol{k}}| ^{2}\overline{s_{\boldsymbol{\lambda }}(\boldsymbol{x}_{\boldsymbol{k}})} s_{\boldsymbol{\nu }}(\boldsymbol{x}_{\boldsymbol{k}})| V(\boldsymbol{x}_{\boldsymbol{k}})| ^{2}.$$
As a consequence the solution formula \eqref{eq:aiSolGen} of Theorem \ref{thm:sol} becomes
$$ a_{i} =\frac{\sum\limits_{j=1}^n (-1)^{i+j}\left(\sum\limits _{\boldsymbol{l} \in \binom{[m]}{n-1} \ }| w_{\boldsymbol{l}}| ^{2}s_{\boldsymbol{\lambda }[ i]}(\boldsymbol{x}_{\boldsymbol{l}})\overline{s_{\boldsymbol{\lambda }[ j]}(\boldsymbol{x}_{\boldsymbol{l}})}| V(\boldsymbol{x}_{\boldsymbol{l}})| ^{2}\right)\left(\sum\limits _{k=1}^{m}| w_{k}| ^{2} \overline{x}_{k}^{d_{j}} y_{k}\right)}{\sum\limits _{\boldsymbol{k} \in \binom{[m]}{n} \ }| w_{\boldsymbol{k}} {s_{\boldsymbol{\lambda }}(\boldsymbol{x}_{\boldsymbol{k}})} V(\boldsymbol{x}_{\boldsymbol{k}})| ^{2}},$$
while the pseudoinverse turns into $ BB^{*}A^{*}W^*W$ with
\begin{align*}
B_{i\boldsymbol{l}} =\frac{( -1)^{i} w_{\boldsymbol{l}}s_{\boldsymbol{\lambda} [i]} (\boldsymbol{x}_{\boldsymbol{l}} )V(\boldsymbol{x}_{\boldsymbol{l}} )}{\sqrt{\sum\limits _{\boldsymbol{k} \in \binom{[m]}{n}} | w_{\boldsymbol{k}}s_{\boldsymbol{\lambda}} (\boldsymbol{x}_{\boldsymbol{k}} )V(\boldsymbol{x}_{\boldsymbol{k}} )|^{2}}}.
\end{align*}

% xxxxxxxxxxxxxxxxxxxxxxxxxxxxxxxxxxxxxxxxxxxxxxxxxxxxxxxxxxxxxxxxxxxxxxxxx

\section{Examples}
\label{sec:ex}

% xxxxxxxxxxxxxxxxxxxxxxxxxxxxxxxxxxxxxxxxxxxxxxxxxxxxxxxxxxxxxxxxxxxxxxxxx

\subsection{The case of polynomial regression}
\label{subsec:PolyRegression}

Consider the case $\boldsymbol{d} = (n-1,n-2,\dotsc ,1,0)$ of polynomial regression of degree $n-1$. In this case, $\lambda_k = 0$ for all $k$, and $\boldsymbol{\lambda}[i] = (\underset{i-1}{\underbrace{1,\ldots,1}},\underset{n-i}{\underbrace{0,\ldots,0}})$. Counting out semistandard Young tableaux, we have $s_{\boldsymbol{\lambda}} = 1$ and $s_{\boldsymbol{\lambda}[i]} = e_{i-1}$, the $(i-1)$th  \emph{elementary symmetric polynomial} with
\begin{align*}
e_{i-1}(x_1,x_2,\ldots,x_{n-1}) = \sum\limits _{\boldsymbol{l} \in \binom{[n-1]}{i-1} \ }\left( \prod _{j=1}^{i-1} x_{l_{j}}\right).
\end{align*}
Thus Equation \eqref{eq:aiSolGen} becomes
\begin{align}
\label{eq:aiSolPolyRegression}
a_i=\frac{\sum\limits_{j=1}^n (-1)^{i+j}  \left( \sum\limits_{\boldsymbol l\in {[m]\choose {n-1}}}e_{i-1}(\boldsymbol x_{\boldsymbol l})\overline{e_{j-1}({ \boldsymbol x}_{\boldsymbol l})}
|V(\boldsymbol x_{\boldsymbol l})|^2 \right) \left(\sum\limits_{k=1}^m \overline x_k^{n-j}  y_k \right)
}{\sum\limits_{\boldsymbol k\in {[m]\choose n}}
|V(\boldsymbol x_{\boldsymbol k})|^2}.
\end{align}

% xxxxxxxxxxxxxxxxxxxxxxxxxxxxxxxxxxxxxxxxxxxxxxxxxxxxxxxxxxxxxxxxxxxxxxxxx

\subsection{Even polynomials of degree $\mathbf 4$}
\label{subsec:EvenPolyDeg4}

Consider the case $\boldsymbol{d} = (4, 2, 0)$ of fitting an even polynomial of the form $f(x) = a_1 x^4 + a_2 x^2 + a_3$.
In this case, $\boldsymbol{\lambda} = (2, 1, 0)$; $\boldsymbol{\lambda}[1] = (1,0)$; $\boldsymbol{\lambda}[2] = (3,0)$; and
$\boldsymbol{\lambda}[3] = (3,2)$; from which one computes
\begin{align*}
    s_{\boldsymbol{\lambda}}(x_1,x_2,x_3)
        &=      (x_1 + x_2)(x_1 + x_3)(x_2 + x_3),
    \\
    s_{\boldsymbol{\lambda}[1]}(x_1,x_2)
        &=      x_1 + x_2,
    \\
    s_{\boldsymbol{\lambda}[2]}(x_1,x_2)
        &=      (x_1 + x_2)(x_1^2 + x_2^2),
    \\
    s_{\boldsymbol{\lambda}[3]}(x_1,x_2)
        &=      x_1^2 x_2^2(x_1 + x_2).
\end{align*}
Then Equation~\eqref{eq:aiSolGen} yields
\begin{align*}
    a_1 &=      \frac{\sum\limits_{j=1}^3 (-1)^{j+1}
                \left( \sum\limits_{\boldsymbol l\in {[m]\choose 2}}s_{\boldsymbol \lambda[1]}(\boldsymbol x_{\boldsymbol l})
                    \overline{s_{\boldsymbol \lambda[j]}(\boldsymbol x_{\boldsymbol l})} |V(\boldsymbol x_{\boldsymbol l})|^2 \right)
                \left( \sum\limits_{k=1}^m \overline x_k^{d_j}  y_k  \right) }
                {\sum\limits_{\boldsymbol k\in {[m]\choose 3}} |s_{\boldsymbol \lambda}(\boldsymbol x_{\boldsymbol k})V(\boldsymbol x_{\boldsymbol k})|^2}
%        \\&=    \frac{
%                \sum\limits_{\boldsymbol l\in {[m]\choose 2}} (x_{l_1} + x_{l_2})|x_{l_1} - x_{l_2}|^2
%                    \sum\limits_{j=1}^3 (-1)^{j+1} \overline{s_{\boldsymbol \lambda[j]}(\boldsymbol x_{\boldsymbol l})}
%                    \sum\limits_{k=1}^m \overline x_k^{d_j}  y_k }
%                {\sum\limits_{\boldsymbol k\in {[m]\choose 3}} \Big|\big((x_{k_1}+x_{k_2})(x_{k_1}+x_{k_3})(x_{k_2}+x_{k_3})\big)
%                    \big((x_{k_1}-x_{k_2})(x_{k_1}-x_{k_3})(x_{k_2}-x_{k_3})\big)\Big|^2}
%        \\&=    \frac{
%                \sum\limits_{\boldsymbol l\in {[m]\choose 2}} (x_{l_1} + x_{l_2})|x_{l_1} - x_{l_2}|^2 \Big(
%                    \sum\limits_{k=1}^m \big(
%                        (\overline{x}_{l_1} + \overline{x}_{l_2})\overline{x}_k^4  y_k
%                        - (\overline{x}_{l_1} + \overline{x}_{l_2})(\overline{x}_{l_1}^2 + \overline{x}_{l_2}^2) \overline{x}_k^2  y_k
%                        + (\overline{x}_{l_1} + \overline{x}_{l_2})\overline{x}_{l_1}^2 \overline{x}_{l_2}^2 y_k \big) \Big) }
%                {\sum\limits_{\boldsymbol k\in {[m]\choose 3}} \big| (x_{k_1}^2 - x_{k_2}^2)(x_{k_1}^2 - x_{k_3}^2)(x_{k_2}^2 - x_{k_3}^2)\big|^2}
        \\&=    \frac{
                \sum\limits_{\boldsymbol l\in {[m]\choose 2}}
                    |x_{l_1}^2 - x_{l_2}^2|^2
                    \sum\limits_{k=1}^m \Big(
                        \overline{x}_k^4  y_k
                        - (\overline{x}_{l_1}^2 + \overline{x}_{l_2}^2) \overline{x}_k^2  y_k
                        + \overline{x}_{l_1}^2 \overline{x}_{l_2}^2  y_k \Big) }
                {\sum\limits_{\boldsymbol k\in {[m]\choose 3}} \big| (x_{k_1}^2 - x_{k_2}^2)(x_{k_1}^2 - x_{k_3}^2)(x_{k_2}^2 - x_{k_3}^2)\big|^2},
\end{align*}
\begin{align*}
    a_2 &=      \frac{\sum\limits_{j=1}^3 (-1)^{j}
                \left( \sum\limits_{\boldsymbol l\in {[m]\choose 2}}s_{\boldsymbol \lambda[2]}(\boldsymbol x_{\boldsymbol l})
                    \overline{s_{\boldsymbol \lambda[j]}(\boldsymbol x_{\boldsymbol l})} |V(\boldsymbol x_{\boldsymbol l})|^2 \right)
                \left( \sum\limits_{k=1}^m \overline{x}_k^{d_j}  y_k \right) }
                {\sum\limits_{\boldsymbol k\in {[m]\choose 3}} |s_{\boldsymbol \lambda}(\boldsymbol x_{\boldsymbol k})V(\boldsymbol x_{\boldsymbol k})|^2}
%    \\&=        \frac{
%                \sum\limits_{\boldsymbol l\in {[m]\choose 2}}(x_{l_1} + x_{l_2})(x_{l_1}^2 + x_{l_2}^2)|x_{l_1} - x_{l_2}|^2
%                    \sum\limits_{j=1}^3 (-1)^{j} \overline{s_{\boldsymbol \lambda[j]}(\boldsymbol x_{\boldsymbol l})}
%                    \sum\limits_{k=1}^m \overline{x}_k^{d_j}  y_k }
%                {\sum\limits_{\boldsymbol k\in {[m]\choose 3}} \big| (x_{k_1}^2 - x_{k_2}^2)(x_{k_1}^2 - x_{k_3}^2)(x_{k_2}^2 - x_{k_3}^2)\big|^2}
%    \\&=        \frac{
%                \sum\limits_{\boldsymbol l\in {[m]\choose 2}}(x_{l_1} + x_{l_2})(x_{l_1}^2 + x_{l_2}^2)|x_{l_1} - x_{l_2}|^2
%                    (\overline{x}_{l_1} + \overline{x}_{l_2}) \Big(
%                    \sum\limits_{k=1}^m - \big(
%                        \overline{x}_k^4 y_k
%                        - (\overline{x}_{l_1}^2 + \overline{x}_{l_2}^2) \overline{x}_k^2 y_k
%                        + \overline{x}_{l_1}^2 \overline{x}_{l_2}^2 y_k
%                    \big) \Big) }
%                {\sum\limits_{\boldsymbol k\in {[m]\choose 3}} \big| (x_{k_1}^2 - x_{k_2}^2)(x_{k_1}^2 - x_{k_3}^2)(x_{k_2}^2 - x_{k_3}^2)\big|^2}
    \\&=        \frac{-
                \sum\limits_{\boldsymbol l\in {[m]\choose 2}}|x_{l_1}^2 - x_{l_2}^2|^2 (x_{l_1}^2 + x_{l_2}^2)
                    \sum\limits_{k=1}^m \Big(
                        \overline{x}_k^4 y_k
                        - (\overline{x}_{l_1}^2 + \overline{x}_{l_2}^2) \overline{x}_k^2 y_k
                        + \overline{x}_{l_1}^2 \overline{x}_{l_2}^2 y_k
                    \Big)  }
                {\sum\limits_{\boldsymbol k\in {[m]\choose 3}} \big| (x_{k_1}^2 - x_{k_2}^2)(x_{k_1}^2 - x_{k_3}^2)(x_{k_2}^2 - x_{k_3}^2)\big|^2},
\end{align*}
and
\begin{align*}
    a_3 &=      \frac{\sum\limits_{j=1}^3 (-1)^{j+1}
                \left( \sum\limits_{\boldsymbol l\in {[m]\choose 2}}s_{\boldsymbol \lambda[3]}(\boldsymbol x_{\boldsymbol l})
                    \overline{s_{\boldsymbol \lambda[j]}(\boldsymbol x_{\boldsymbol l})} |V(\boldsymbol x_{\boldsymbol l})|^2 \right)
                \left( \sum\limits_{k=1}^m \overline{x}_k^{d_j}  y_k \right) }
                {\sum\limits_{\boldsymbol k\in {[m]\choose 3}} |s_{\boldsymbol \lambda}(\boldsymbol x_{\boldsymbol k})V(\boldsymbol x_{\boldsymbol k})|^2}
%    \\&=        \frac{\sum\limits_{\boldsymbol l\in {[m]\choose 2}} x_{l_1}^2 x_{l_2}^2 (x_{l_1} + x_{l_2}) |x_{l_1} - x_{l_2}|^2
%                    \sum\limits_{j=1}^3 (-1)^{j+1} \overline{s_{\boldsymbol \lambda[j]}(\boldsymbol x_{\boldsymbol l})}
%                    \sum\limits_{k=1}^m \overline x_k^{d_j}  y_k }
%                {\sum\limits_{\boldsymbol k\in {[m]\choose 3}} |s_{\boldsymbol \lambda}(\boldsymbol x_{\boldsymbol k})V(\boldsymbol x_{\boldsymbol k})|^2}
%    \\&=        \frac{ \sum\limits_{\boldsymbol l\in {[m]\choose 2}} x_{l_1}^2 x_{l_2}^2 (x_{l_1} + x_{l_2}) |x_{l_1} - x_{l_2}|^2
%                    (\overline{x}_{l_1} + \overline{x}_{l_2})
%                    \Big( \sum\limits_{k=1}^m \big(
%                        \overline{x}_k^4 y_k
%                        - (\overline{x}_{l_1}^2 + \overline{x}_{l_2}^2) \overline{x}_k^2 y_k
%                        + \overline{x}_{l_1}^2 \overline{x}_{l_2}^2 y_k
%                    \big) \Big)}
%                {\sum\limits_{\boldsymbol k\in {[m]\choose 3}} \big| (x_{k_1}^2 - x_{k_2}^2)(x_{k_1}^2 - x_{k_3}^2)(x_{k_2}^2 - x_{k_3}^2)\big|^2}
    \\&=        \frac{\sum\limits_{\boldsymbol l\in {[m]\choose 2}} |x_{l_1}^2 - x_{l_2}^2|^2 x_{l_1}^2 x_{l_2}^2
                    \sum\limits_{k=1}^m \Big(
                        \overline{x}_k^4 y_k
                        - (\overline{x}_{l_1}^2 + \overline{x}_{l_2}^2) \overline{x}_k^2 y_k
                        + \overline{x}_{l_1}^2 \overline{x}_{l_2}^2 y_k
                    \Big) }
                {\sum\limits_{\boldsymbol k\in {[m]\choose 3}} \big| (x_{k_1}^2 - x_{k_2}^2)(x_{k_1}^2 - x_{k_3}^2)(x_{k_2}^2 - x_{k_3}^2)\big|^2}.
\end{align*}
The matrix entry $B_{i\boldsymbol{l}}$ is given by
\begin{align*}
    B_{i\boldsymbol{l}}
%    &=          \frac{(-1)^{i} s_{\boldsymbol{\lambda} [i]} (\boldsymbol{x}_{\boldsymbol{l}} )V(\boldsymbol{x}_{\boldsymbol{l}} )}
%                {\sqrt{\sum\limits_{\boldsymbol{k} \in \binom{[m]}{3}}
%                    |s_{\boldsymbol{\lambda}} (\boldsymbol{x}_{\boldsymbol{k}}) V(\boldsymbol{x}_{\boldsymbol{k}} )|^{2}}}
%    \\&=
    =\frac{(-1)^{i} (x_{l_1} - x_{l_2}) s_{\boldsymbol{\lambda} [i]} (\boldsymbol{x}_{\boldsymbol{l}})}
                {\sqrt{\sum\limits_{\boldsymbol{k} \in \binom{[m]}{3}}
                    \big| (x_{k_1}^2 - x_{k_2}^2)(x_{k_1}^2 - x_{k_3}^2)(x_{k_2}^2 - x_{k_3}^2)\big|^2 }},
\end{align*}
i.e.,
\begin{align*}
    B_{1,\boldsymbol{l}}
    &=          \frac{ x_{l_2}^2 - x_{l_1}^2 }
                {\sqrt{\sum\limits_{\boldsymbol{k} \in \binom{[m]}{3}}
                    \big| (x_{k_1}^2 - x_{k_2}^2)(x_{k_1}^2 - x_{k_3}^2)(x_{k_2}^2 - x_{k_3}^2)\big|^2 }},
    \\
    B_{2,\boldsymbol{l}}
    &=          \frac{ x_{l_1}^4 - x_{l_2}^4 }
                {\sqrt{\sum\limits_{\boldsymbol{k} \in \binom{[m]}{3}}
                    \big| (x_{k_1}^2 - x_{k_2}^2)(x_{k_1}^2 - x_{k_3}^2)(x_{k_2}^2 - x_{k_3}^2)\big|^2 }},
    \\
    B_{3,\boldsymbol{l}}
    &=          \frac{ x_{l_1}^2 x_{l_2}^2 (x_{l_2}^2 - x_{l_1}^2)}
                {\sqrt{\sum\limits_{\boldsymbol{k} \in \binom{[m]}{3}}
                    \big| (x_{k_1}^2 - x_{k_2}^2)(x_{k_1}^2 - x_{k_3}^2)(x_{k_2}^2 - x_{k_3}^2)\big|^2 }}.
\end{align*}

% xxxxxxxxxxxxxxxxxxxxxxxxxxxxxxxxxxxxxxxxxxxxxxxxxxxxxxxxxxxxxxxxxxxxxxxxx

\subsection{Fitting power functions}
\label{subsec:Power}

Suppose $\boldsymbol{d} = (d)$, i.e., the function to fit is the power function $f(x) = ax^d$.
Then $\boldsymbol{\lambda} = \boldsymbol{d} = (d)$ so that $s_{\boldsymbol{\lambda}}(x) = x^d$, and $\boldsymbol{d}\langle 1 \rangle = \boldsymbol{\lambda}[1]$
is the empty partition. Equation~\eqref{eq:aiSolGen} becomes
\[
    a   %=   \frac{\sum\limits_{k=1}^m \overline x_k^{d} y_k}
       %     {\sum\limits_{k=1}^n x_k^d \overline{x}_k^d}
        =   \frac{\sum\limits_{k=1}^m \overline x_k^{d} y_k}
            {\sum\limits_{k=1}^m | x_k|^{2d} }.
\]

% xxxxxxxxxxxxxxxxxxxxxxxxxxxxxxxxxxxxxxxxxxxxxxxxxxxxxxxxxxxxxxxxxxxxxxxxx
% xxxxxxxxxxxxxxxxxxxxxxxxxxxxxxxxxxxxxxxxxxxxxxxxxxxxxxxxxxxxxxxxxxxxxxxxx
% xxxxxxxxxxxxxxxxxxxxxxxxxxxxxxxxxxxxxxxxxxxxxxxxxxxxxxxxxxxxxxxxxxxxxxxxx

%\section{Proof of Theorem \ref{thm:minors}}
%\label{sec:Proof}
%$\left(\overline x_k^{\lambda_i+i-1}  x_k^{\nu_j+j-1}\right)_{i,j=0,1,\dots,n}=:( v_{k\lambda \nu_0}, v_{k\lambda \nu_1},\dots, v_{k\lambda \nu_n})$

%$ v_{k\lambda \nu}=x_k^\nu v_{k\lambda 0}$

%\begin{align*}
%\left|\left(\sum\limits_{k=1}^m\overline x_k^{\lambda_i+i-1}  x_k^{\nu_j+j-1}\right)_{i,j=0,1,\dots,n}\right|&=\sum\limits_{k_0,\dots,k_n}\left|( v_{k_0\lambda \nu_0-1}, v_{k_1\lambda \nu_1},\dots, v_{k_n\lambda \nu_{n}+n-1})\right|\\
%&=\sum\limits_{k_0,\dots,k_n \mbox{ \footnotesize{distinct}}}\left|( v_{k_0\lambda \nu_0-1}, v_{k_1\lambda \nu_1},\dots, v_{k_m\lambda \nu_{n}+n-1})\right|\\
%&=\sum\limits_{1\le k_0<\dots<k_n\le m}\left|\left(\sum\limits_{l} v_{k_l\lambda \nu_0-1},\sum\limits_{l} v_{k_l\lambda \nu_1},\dots,\sum\limits_{l} v_{k_l\lambda \nu_{n}+n-1}\right)\right|\\
%&=\sum\limits_{1\le k_0<\dots<k_n\le m}\left|\left(\sum\limits_{l}^m\overline x_{k_l}^{\lambda_i+i-1}  x_{k_l}^{\nu_j+j-1}\right)_{i,j=0,1,\dots,n}\right|
%\end{align*}

% xxxxxxxxxxxxxxxxxxxxxxxxxxxxxxxxxxxxxxxxxxxxxxxxxxxxxxxxxxxxxxxxxxxxxxxxx
% xxxxxxxxxxxxxxxxxxxxxxxxxxxxxxxxxxxxxxxxxxxxxxxxxxxxxxxxxxxxxxxxxxxxxxxxx
% xxxxxxxxxxxxxxxxxxxxxxxxxxxxxxxxxxxxxxxxxxxxxxxxxxxxxxxxxxxxxxxxxxxxxxxxx

\section{Complexity}
\label{sec:Eff}

% xxxxxxxxxxxxxxxxxxxxxxxxxxxxxxxxxxxxxxxxxxxxxxxxxxxxxxxxxxxxxxxxxxxxxxxxx
% xxxxxxxxxxxxxxxxxxxxxxxxxxxxxxxxxxxxxxxxxxxxxxxxxxxxxxxxxxxxxxxxxxxxxxxxx
% xxxxxxxxxxxxxxxxxxxxxxxxxxxxxxxxxxxxxxxxxxxxxxxxxxxxxxxxxxxxxxxxxxxxxxxxx

Here we examine the complexity of possible evaluations of Equation \eqref{eq:pseudo}. The simple message is that, for fixed $ n$  and $d_1$, it can be estimated asymptotically as $ \mathcal{O}( m^{n})$ for $ m\rightarrow \infty $ (for an introduction to asymptotic analysis the reader may consult \cite{concrete}).
The dominating contributions come from the matrix multiplication and the evaluation of $|A^*A|$ in the denominators. The coefficients of the estimates depend on $ n$ and $ d_{1}$, and the details of these dependencies depend on the details of the used evaluation method. The presentation is hopefully intelligible for people that are not accustomed to counting floating point operations (FLOPs).
We have to conclude that for $n>3$ evaluating our formulas is not competitive with the numerical algorithms commonly used.

First, note that the complexity of naively evaluating the product $XY$ of two matrices $ X\in \mathbb{C}^{m\times n}$ and $Y\in \mathbb{C}^{n\times p}$ is $ \operatorname{\mathcal{O}}(mnp)$. So the complexity of naively multiplying $ (BB^{*}) A^{*}$ left-to-right is $ \operatorname{\mathcal{O}}\left( n^2\binom{m}{n-1}+n^2m\right) =\operatorname{\mathcal{O}}\left( n^2\binom{m}{n-1}\right)$ as $ B\in \mathbb{C}^{n\times \binom{m}{n-1}} ,B^{*} \in \mathbb{C}^{\binom{m}{n-1} \times n} ,A^{*} \in \mathbb{C}^{n\times m}$.
For positive integers $ M,N$ we have the well-known inequality $ \left(\frac{M}{N}\right)^{N} \leq \binom{M}{N} \leq \left(\frac{eM}{N}\right)^{N}$; see \cite[Part VIII, Appendix C]{Cormen}. In particular, we can estimate
$$ n^2\binom{m}{n-1} \leq n^2\left(\frac{em}{n-1}\right)^{n-1}.$$
%For $ n\ll m$ this in turn can be estimated by $ \operatorname{\mathcal{O}}\left( nm\binom{m}{n-1}\right) =\operatorname{\mathcal{O}}\left( n\left(\frac{e}{n-1}\right)^{n-1} m^{n}\right)$.
Now we use that for $ m\gg n\gg 0$ we have% $ n\left(\frac{e}{n-1}\right)^{n-1} m^{n} \sim \frac{n^{2}}{e}\left(\frac{em}{n}\right)^{n}$ . This is because
% \begin{align*}
% n^{n} &=( n-1)^{n}\left(\frac{n}{n-1}\right)^{n} =( n-1)( n-1)^{n-1}\left(\frac{1}{1-1/n}\right)^{n} \\
% &\sim \frac{( n-1)( n-1)^{n-1}}{e},
% \end{align*}
% which implies $ \frac{1}{( n-1)^{n-1}} \sim \frac{( n-1)}{en^{n}}$.
 \begin{align} \label{eq:binom}
\left(\frac{em}{n-1}\right)^{n-1} = \left(\frac{em}{n}\right)^{n-1}\left(\frac{n}{n-1}\right)^{n-1} \sim e\left(\frac{em}{n}\right)^{n-1}.
 \end{align}
From this, we deduce $ \operatorname{\mathcal{O}}\left( n^2\binom{m}{n-1}\right) =\operatorname{\mathcal{O}}\left(en^{2}\left(\frac{em}{n}\right)^{n-1}\right) =\operatorname{\mathcal{O}}\left(n^{2}\left(\frac{em}{n}\right)^{n-1}\right)$.

Now we would like to investigate the complexity of evaluating the matrix $ B$ of Equation \eqref{eq:B}.
In the denominator we have to evaluate $s_{\boldsymbol{\lambda }} (\boldsymbol{x}_{\boldsymbol{k}} )$. We could do that using a multivariate Horner's scheme for evaluating a polynomial $ f$ in $n$ variables, the number of FLOPs being $ 2\binom{\deg( f) +n}{n} -2$; see \cite[Theorem 3.1]{Czekansky}. Hence with $\deg( s_{\boldsymbol{\lambda }} (\boldsymbol{x}_{\boldsymbol{k}} )) =|\boldsymbol{\lambda } |=\sum _{k=1}^{n}( d_{k} +k-n) =\binom{n+1}{2} -n^{2} +\sum _{k=1}^{n} d_{k} =-\binom{n}{2} +\sum _{k=1}^{n} d_{k}$ this is $ 2\binom{|\boldsymbol{\lambda } |+n}{n} -2$, which is neither a polynomial in $|\boldsymbol{\lambda } |$ nor $d_1$ nor $n$. % But $ 2\binom{|\boldsymbol{\lambda } |+m}{m} -2=\mathcal{O}\left(( |\boldsymbol{\lambda } |+m)^{m}\right)$ grows extremely rapidly with the number of data points $ m$.
A better evaluation method for $s_{\boldsymbol{\lambda }} (\boldsymbol{x}_{\boldsymbol{k}} )$ seems to be the second Jacobi-Trudi identity (see \cite[Equation (3.5)]{Macdonald}), expressing $s_{\boldsymbol{\lambda }} (\boldsymbol{x}_{\boldsymbol{k}} )$ as the determinant of a $ \lambda _{1} \times \lambda _{1}$ matrix in the elementary symmetric polynomials $e_i (\boldsymbol{x}_{\boldsymbol{k}} )$. To evaluate these elementary symmetric polynomials one has to expand
\[
( X-x_{k_1})( X-x_{k_2}) \cdots ( X-x_{k_n}) = X^n-e_1 (\boldsymbol{x}_{\boldsymbol{k}} )X^{n-1}+e_2 (\boldsymbol{x}_{\boldsymbol{k}} )X^{n-2}-e_3 (\boldsymbol{x}_{\boldsymbol{k}} )X^{n-3}\pm \ldots +(-1)^n e_n (\boldsymbol{x}_{\boldsymbol{k}} ).
\]
If $ ( X-x_{k_1})( X-x_{k_2}) \cdots ( X-x_{k_{i-1}})$ is already expanded, then expanding
$ ( X-x_{k_1})( X-x_{k_2}) \cdots ( X-x_{k_{i-1}}) \ \cdot \ ( X-x_{k_{i}})$ just needs
$ i-1$ extra multiplications and $ i-1$ extra additions, i.e., \ $ 2( i-1)$ more FLOPs.
Hence evaluating the elementary symmetric polynomials costs
$$ \sum\limits _{i=1}^{n}2(i-1) =n(n-1)=\mathcal{O}( n^{2})$$
FLOPs. Calculating the determinant costs in the worst case $ \mathcal{O}( \lambda _{1}^{3})$ FLOPs (see \cite{Higham}). For the specific Schur polynomials in question $ \lambda _{1} = d_{1} + 1 - n$, and the cost of evaluating such a Schur polynomial can be estimated as $ \mathcal{O}( n^{2} + (d_{1} + 1 - n)^{3})$.
Evaluating the Vandermonde determinant $V (\boldsymbol{x}_{\boldsymbol{k}} )$ costs $2\binom{n}{2}-1=\mathcal{O}(n^2)$ FLOPs, and hence evaluating %all of those in the denominator by diving out linear factors $\binom{m}{2}+2m(m-n)=\mathcal{O}(m^2)$ FLOPs.
each term $s_{\boldsymbol{\lambda }} (\boldsymbol{x}_{\boldsymbol{k}}) V (\boldsymbol{x}_{\boldsymbol{k}} )$ in the denominator of Equation \eqref{eq:B} costs $ \mathcal{O}( 2n^{2} + (d_{1} + 1 - n)^{3})$ FLOPs.
So the complexity of evaluating the full denominator is
$$ \mathcal{O}\left( \binom{m}{n} \left(2n^{2} + (d _{1} + 1 - n)^{3}\right)\right) =\mathcal{O}\left( \left(\frac{em}{n}\right)^{n} \left(n^{2} + (d _{1} + 1 - n)^{3}\right)\right),$$
%For $ m\gg n, d_1 \geq 2$ the cost of it is roughly the same as
which is larger than
that of the matrix multiplication. The numerator of Equation \eqref{eq:B} can be analyzed along similar lines, yielding with \eqref{eq:binom} an estimate
%$$ \mathcal{O}\left(m^2+ n\binom{m}{n-1} (d_{1}+n)^{3}\right) =\mathcal{O}\left(m^2+n\left(\frac{e}{n-1}\right)^{n-1} m^{n-1} (d_{1}+n)^{3}\right).$$
\begin{align*}
 & \mathcal{O}\left(n\binom{m}{n-1} \left(3(n-1)^2+ 2(d_{1}+2-n)^{3}\right)+nm(d_1+1)\right)\\ & =\mathcal{O}\left(ne\left(\frac{em}{n}\right)^{n-1} \left(3(n-1)^2+ 2(d_{1}+2-n)^{3}\right) \right)\\
 & =\mathcal{O}\left(n\left(\frac{em}{n}\right)^{n-1} \left((n-1)^2+ (d_{1}+2-n)^{3}\right) \right).
\end{align*}
% Note here that the evaluations of the elementary symmetric polynomials can be repurposed, hence they contribute $m(3m+2)=\mathcal{O}(m^2)$ FLOPs.
For $ m\gg n > 2$ the cost of evaluating the numerator is negligible in comparison to the contribution of the denominator.

% xxxxxxxxxxxxxxxxxxxxxxxxxxxxxxxxxxxxxxxxxxxxxxxxxxxxxxxxxxxxxxxxxxxxxxxxx

\subsection{A numerical example}
\label{subsec:num}
Let us look into a concrete example of the evaluation that we implemented in Mathematica \cite{mathematica} for $\boldsymbol{d} =( 4,2,0)$, i.e., the model function to fit is again $f( x) =a_{1} x^{4} +a_{2} x^{2} +a_{3}$ with the $a_i\in\R$.

To give a practical application, this could, for example, be taken as a model function used for digital image processing of a human mouth. Let us assume for the moment $f(x)$ to be non-constant. From $f'( x) =4a_{1} x^{3} +2a_{2} x=2x\left( 2a_{1} x^{2} +a_{2}\right)$ we deduce that for $a_{1} a_{2} \geq 0$ there is just one extremum. Otherwise there are three extrema $x=0$ and $x=\pm \sqrt{-a_{2}/(2a_{1})}$. %Here we take the liberty to ignore the non-generic case $a_{1} =0$.
From the sign of $a_{1} a_{2}$ a machine could infer if the mouth is smiling (i.e., $a_{1} a_{2}  \ge 0$) or frowning (i.e., $a_{1} a_{2} < 0$)
 \footnote{In physics this type of model functions plays a role in the study of phase transitions associated to the phenomenon of spontaneous symmetry breaking.}.

 To investigate empirically the running time of the evaluation of our Formula \eqref{eq:aiSolGen} we created data points by superimposing random noise on the evaluations of
\begin{align}\label{eq:numfunc}
g(x) =x^{4} -\frac{5}{2} \cdot 10^{5} \ x^{2}.
\end{align}
The range for $x$ is here $[ -500,500] \cap \ell\mathbb{Z}$. In Figure \ref{fig:101pts} we show an example of a regression obtained this way with $\ell=10$, i.e., $m=101$ data points.
\begin{figure}[!htbp]
    \centering
    \includegraphics[width=0.8\textwidth]{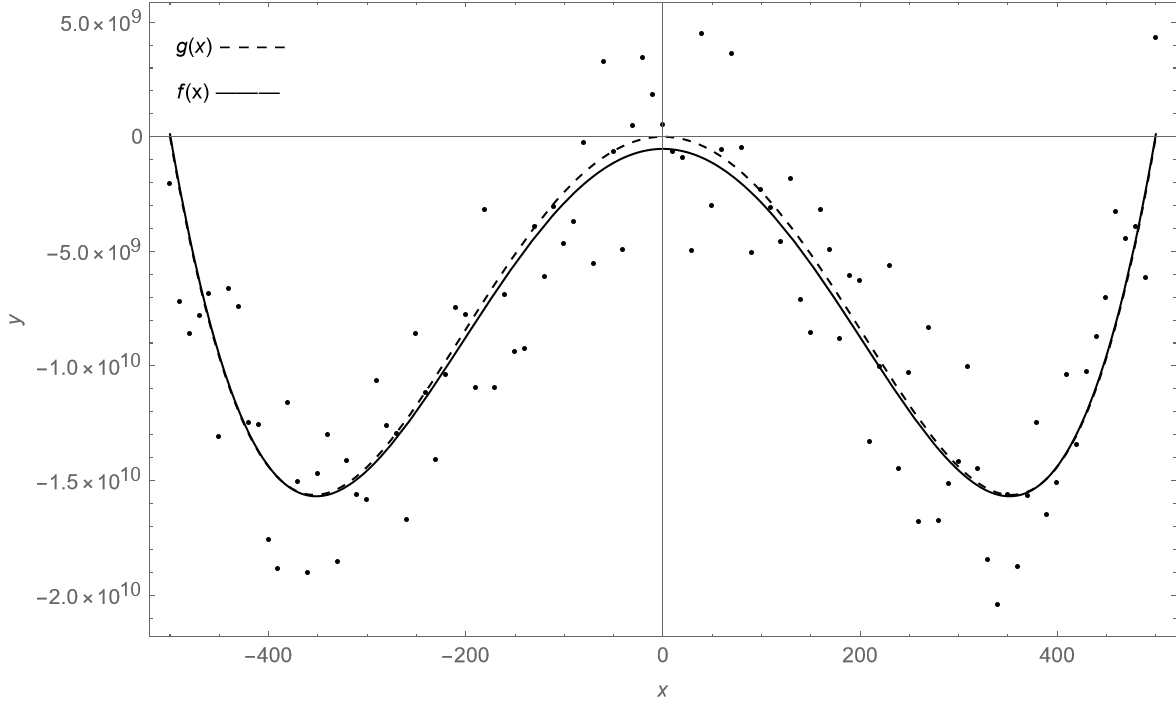}
    \caption{
Our numerical example with $\ell=10$. The dashed graph is the one of the function $g( x)$ of Equation \eqref{eq:numfunc}, the solid one is the graph of the resulting regression $f(x)$.}
    \label{fig:101pts}
\end{figure}
In Figure \ref{fig:loglog} we depict how the running time of our implementation rises when we  increase the number $m$ of data points.
\begin{figure}[!htbp]
    \centering
    \includegraphics[width=0.6\textwidth]{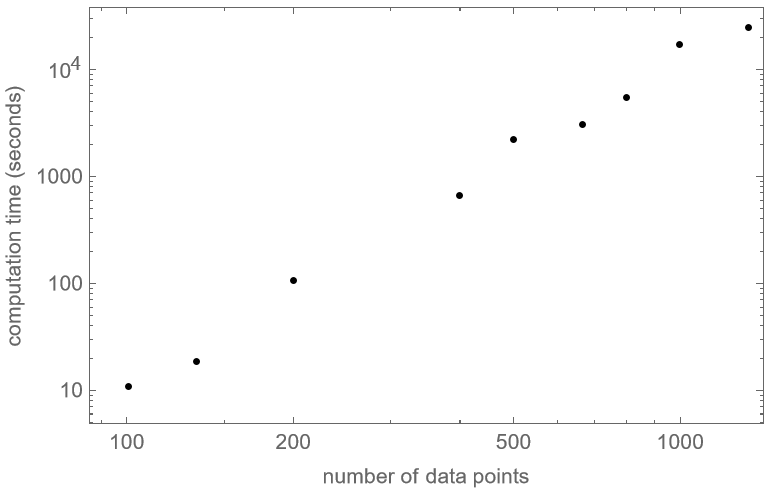}
    \caption{Here we put a log-log plot of the running time (measured in seconds) of our implementation on the $y$-axis against the number of data points on the $x$-axis. The slope equals approximately $3$, verifying empirically our predicted complexity of the evaluation as $\mathcal{O}(m^3)$.}
    \label{fig:loglog}
\end{figure}

% xxxxxxxxxxxxxxxxxxxxxxxxxxxxxxxxxxxxxxxxxxxxxxxxxxxxxxxxxxxxxxxxxxxxxxxxx

\section{Recursive evaluations}
\label{sec:Realtime}

Observe that, when fitting a fixed polynomial $f(x) =\sum_{i=1}^{n} a_{i} x^{d_{i}}$ to a large data set or multiple data sets,
the partitions $\boldsymbol \lambda$ and $\boldsymbol \lambda[i]$, $i=1,\ldots,n$, depend only on $\boldsymbol d$. Therefore, the polynomials $s_{\boldsymbol \lambda[i]}$ and
$s_{\boldsymbol \lambda}$ can be computed (e.g., using the second Jacobi-Trudi identity) and stored once. In addition, if a data set is
enlarged, Equation~\eqref{eq:aiSolGen} can be used recursively as we now describe. In this way, amending one additional data point
to a data set of $m$ data points reduces the complexity of the evaluations from $\mathcal O(m^n)$ to $\mathcal O(m^{n-1})$ for large $m$.

Suppose that one has applied Equation~\eqref{eq:aiSolGen} to a collection of data points $\boldsymbol{x} = (x_1,x_2,\ldots,x_m)^T$ and
$\boldsymbol{y} = (y_1,y_2,\ldots,y_m)^T$. In this process, one stores not only the $a_i$, but the following quantities.
For $i = 1, \ldots, n$ and $j = 1,\ldots, n$, define
\[
    \mathcal{S}_{i,j}(\boldsymbol{x})
        =   \sum\limits_{\boldsymbol l\in {[m]\choose {n-1}}}s_{\boldsymbol \lambda[i]}(\boldsymbol x_{\boldsymbol l})\overline{s_{\boldsymbol \lambda[j]}
            ( \boldsymbol x_{\boldsymbol l})} |V(\boldsymbol x_{\boldsymbol l})|^2,
\]
\[
    \mathcal{T}_{j}(\boldsymbol{x},\boldsymbol{y})
        =   \sum\limits_{k=1}^m \overline x_k^{d_j}  y_k,
\]
and let
\[
    \mathcal{N}_i(\boldsymbol{x},\boldsymbol{y})
        =   \sum\limits_{j=1}^n (-1)^{i+j}\mathcal{S}_{i,j}(\boldsymbol{x})
            \mathcal{T}_{j}(\boldsymbol{x},\boldsymbol{y})
\]
and
\[
    \mathcal{D}(\boldsymbol{x})
        =   \sum\limits_{\boldsymbol k\in {[m]\choose n}} |s_{\boldsymbol \lambda}(\boldsymbol x_{\boldsymbol k}) V(\boldsymbol x_{\boldsymbol k})|^2
\]
denote the numerators and denominator in the expression for $a_i$ in Equation~\eqref{eq:aiSolGen}.
Enlarging the data set by one point, let
$\boldsymbol{x}^\prime = (x_1, x_2, \ldots, x_m, x_{m+1})^T$ and $\boldsymbol{y}^\prime = (y_1, y_2, \ldots, y_m, y_{m+1})^T$.
We express $\mathcal{N}_i(\boldsymbol{x}^\prime, \boldsymbol{y}^\prime)$ as
\begin{align*}
            \sum\limits_{j=1}^n & (-1)^{i+j}
                \left(\sum\limits_{\boldsymbol l\in {[m+1]\choose {n-1}}}s_{\boldsymbol \lambda[i]}(\boldsymbol x_{\boldsymbol l})\overline{s_{\boldsymbol \lambda[j]}
                    ( \boldsymbol x_{\boldsymbol l})} |V(\boldsymbol x_{\boldsymbol l})|^2\right)\left(\sum\limits_{k=1}^{m+1} \overline x_k^{d_j}  y_k\right)
    \\&=    \sum\limits_{j=1}^n (-1)^{i+j}
                \left(\mathcal{S}_{i,j}(\boldsymbol{x},\boldsymbol{y}) +
                \sum\limits_{\boldsymbol l\in {[m]\choose {n-2}}}s_{\boldsymbol \lambda[i]}(\boldsymbol x_{\boldsymbol l,m+1})\overline{s_{\boldsymbol \lambda[j]}
                    ( \boldsymbol x_{\boldsymbol l,m+1})} |V(\boldsymbol x_{\boldsymbol l,m+1})|^2 \right)\left( \mathcal{T}_{j}(\boldsymbol{x},\boldsymbol{y}) +\overline x_{m+1}^{d_j}  y_{m+1} \right)
    \\&=    \mathcal{N}_i(\boldsymbol{x},\boldsymbol{y}) +
            \sum\limits_{j=1}^n (-1)^{i+j} \Bigg( \mathcal{S}_{i,j}(\boldsymbol{x},\boldsymbol{y}) \overline x_{m+1}^{d_j}  y_{m+1} \\
            &\qquad\qquad\quad +
            \!\!\!\sum\limits_{\boldsymbol l\in {[m]\choose {n-2}}} \!\! s_{\boldsymbol \lambda[i]}(\boldsymbol x_{\boldsymbol l,m+1})\overline{s_{\boldsymbol \lambda[j]}
                    ( \boldsymbol x_{\boldsymbol l,m+1})} |V(\boldsymbol x_{\boldsymbol l,m+1})|^2
               %\! \!\sum\limits_{\boldsymbol l\in {[m]\choose {n-2}}} \!\!
               \left( \mathcal{T}_{j}(\boldsymbol{x},\boldsymbol{y}) +\overline x_{m+1}^{d_j}  y_{m+1} \right)\Bigg).
\end{align*}
Let
\[
    \mathcal{R}_{i,j}(\boldsymbol{x}^\prime)%, \boldsymbol{y}^\prime
        =   \sum\limits_{\boldsymbol l\in {[m]\choose {n-2}}}s_{\boldsymbol \lambda[i]}(\boldsymbol x_{\boldsymbol l,m+1})\overline{s_{\boldsymbol \lambda[j]}
                    ( \boldsymbol x_{\boldsymbol l,m+1})} |V(\boldsymbol x_{\boldsymbol l,m+1})|^2,
\]
and then we can express $\mathcal{N}_i(\boldsymbol{x}^\prime, \boldsymbol{y}^\prime)$ as
\[
    \mathcal{N}_i(\boldsymbol{x},\boldsymbol{y}) +
            \sum\limits_{j=1}^n (-1)^{i+j} \bigg(\big(\mathcal{S}_{i,j}(\boldsymbol{x},\boldsymbol{y}) +
                    \mathcal{R}_{i,j}(\boldsymbol{x}^\prime) \big)\overline x_{m+1}^{d_j}  y_{m+1}
                + \mathcal{R}_{i,j}(\boldsymbol{x}^\prime)\mathcal{T}_{j}(\boldsymbol{x},\boldsymbol{y})\bigg).
\]
Similarly, $\mathcal{D}(\boldsymbol{x}^\prime)$ can be written
\begin{align*}
           \sum\limits_{\boldsymbol k\in {[m+1]\choose n}} |s_{\boldsymbol \lambda}(\boldsymbol x_{\boldsymbol k}) V(\boldsymbol x_{\boldsymbol k})|^2
    =    \mathcal{D}(\boldsymbol{x})
            + \sum\limits_{\boldsymbol k\in {[m]\choose n-1}} |s_{\boldsymbol \lambda}(\boldsymbol x_{\boldsymbol k,m+1}) V(\boldsymbol x_{\boldsymbol k,m+1})|^2.
\end{align*}
This yields the following expression for the $a_i^\prime$ to fit to the data set $(\boldsymbol{x}^\prime, \boldsymbol{y}^\prime)$:
{
\begin{equation}
\label{eq:aiSolGen2}
    a_i^\prime
        =   \frac{\mathcal{D} a_i +
            \sum\limits_{j=1}^n (-1)^{i+j}\left( \left(\mathcal{S}_{i,j} +
                    \mathcal{R}_{i,j} \right) \overline x_{m+1}^{d_j}  y_{m+1} +  \mathcal{R}_{i,j} \mathcal{T}_{j}\right)}
            {\mathcal{D} +
                \sum\limits_{\boldsymbol k\in {[m]\choose n-1}} |s_{\boldsymbol \lambda}(\boldsymbol x_{\boldsymbol k,m+1}) V(\boldsymbol x_{\boldsymbol k,m+1})|^2}
\end{equation}}
for $i=1,\dots, n$, where $\mathcal{D} = \mathcal{D}(\boldsymbol{x})$,
$\mathcal{S}_{i,j} = \mathcal{S}_{i,j}(\boldsymbol{x})$, etc., and
$\mathcal{R}_{i,j} = \mathcal{R}_{i,j}(\boldsymbol{x}^\prime)$.

% xxxxxxxxxxxxxxxxxxxxxxxxxxxxxxxxxxxxxxxxxxxxxxxxxxxxxxxxxxxxxxxxxxxxxxxxx

\subsection{Recursive evaluation of even polynomials of degree $\mathbf 4$}
\label{subsec:RealtimeEvenPolyDeg4}

As an illustration, let us consider the case of fitting a polynomial of the form $f(x) = a_1 x^4 + a_2 x^2 + a_3$ treated in
Section~\ref{subsec:EvenPolyDeg4}. Fitting such a polynomial to a data set
$\boldsymbol{x} = (x_1,x_2,\ldots,x_m)^T$ and $\boldsymbol{y} = (y_1,y_2,\ldots,y_m)^T$ involves computing
\begin{align*}
    \mathcal{S}_{1,1}(\boldsymbol{x})
        &=  \sum\limits_{\boldsymbol l\in {[m]\choose 2}} |x_{l_1}^2 - x_{l_2}^2|^2,
    \\
    \mathcal{S}_{1,2}(\boldsymbol{x})
        &= \overline{{\mathcal{S}_{2,1}(\boldsymbol{x})}}
        =\sum\limits_{\boldsymbol l\in {[m]\choose 2}} (\overline{x}_{l_1}^2 + \overline{x}_{l_2}^2) |x_{l_1}^2 - x_{l_2}^2|^2,
    \\
    \mathcal{S}_{1,3}(\boldsymbol{x})
        &= \overline{{\mathcal{S}_{3,1}(\boldsymbol{x})}}=\sum\limits_{\boldsymbol l\in {[m]\choose 2}} \overline{x}_{l_1}^2 \overline{x}_{l_2}^2 |x_{l_1}^2 - x_{l_2}^2|^2, \
    %\mathcal{S}_{2,1}(\boldsymbol{x})
     %   =  \sum\limits_{\boldsymbol l\in {[m]\choose 2}} (x_{l_1}^2 + x_{l_2}^2) |x_{l_1}^2 - x_{l_2}^2|^2,
    \\
    \mathcal{S}_{2,2}(\boldsymbol{x})
        &=  \sum\limits_{\boldsymbol l\in {[m]\choose 2}} |x_{l_1}^4 - x_{l_2}^4|^2, \\
    \mathcal{S}_{2,3}(\boldsymbol{x})&= \overline{{\mathcal{S}_{3,2}(\boldsymbol{x})}}=
          \sum\limits_{\boldsymbol l\in {[m]\choose 2}} \overline{x}_{l_1}^2 \overline{x}_{l_2}^2 (x_{l_1}^2 + x_{l_2}^2) |x_{l_1}^2 - x_{l_2}^2|^2,
    \\
  % \mathcal{S}_{3,1}(\boldsymbol{x})
  %      &=  \sum\limits_{\boldsymbol l\in {[m]\choose 2}} x_{l_1}^2 x_{l_2}^2 |x_{l_1}^2 - x_{l_2}^2|^2, \
  %  \mathcal{S}_{3,2}(\boldsymbol{x})
  %      =  \sum\limits_{\boldsymbol l\in {[m]\choose 2}} x_{l_1}^2 x_{l_2}^2 (\overline{x}_{l_1}^2 + \overline{x}_{l_2}^2) |x_{l_1}^2 - x_{l_2}^2|^2,
  %  \\
    \mathcal{S}_{3,3}(\boldsymbol{x})
        &=  \sum\limits_{\boldsymbol l\in {[m]\choose 2}} |x_{l_1}|^4 |x_{l_2}|^4 |x_{l_1}^2 - x_{l_2}^2|^2,\\
    \mathcal{T}_{j}(\boldsymbol{x},\boldsymbol{y})
        &=  \sum\limits_{k=1}^m \overline x_k^{d_j}  y_k, \quad j=1,2,3; \quad (d_1, d_2, d_3) = (4, 2, 0),\\
    \mathcal{N}_i(\boldsymbol{x},\boldsymbol{y})
       & =   (-1)^i \big( - \mathcal{S}_{i,1}(\boldsymbol{x}) \mathcal{T}_1(\boldsymbol{x},\boldsymbol{y})
                + \mathcal{S}_{i,2}(\boldsymbol{x}) \mathcal{T}_2(\boldsymbol{x},\boldsymbol{y})
                - \mathcal{S}_{i,3}(\boldsymbol{x}) \mathcal{T}_3(\boldsymbol{x},\boldsymbol{y}) \big),
        \qquad i=1,2,3,
\end{align*}
and
\[
    \mathcal{D}(\boldsymbol{x})
        =   \sum\limits_{\boldsymbol k\in {[m]\choose 3}} \big| (x_{k_1}^2 - x_{k_2}^2)(x_{k_1}^2 - x_{k_3}^2)(x_{k_2}^2 - x_{k_3}^2)\big|^2.
\]
Given these quantities, one can fit the data set $\boldsymbol{x}^\prime = (\boldsymbol{x},x_{m+1})$ and
$\boldsymbol{y}^\prime = (\boldsymbol{y},y_{m+1})$ with an additional data point $(x_{m+1},y_{m+1})$ by
computing
\begin{align*}
    \mathcal{R}_{1,1}(\boldsymbol{x}^\prime)
        &=  \sum\limits_{l=1}^m |x_{l}^2 - x_{m+1}^2|^2,\\
        \mathcal{R}_{1,2}(\boldsymbol{x}^\prime)
        &=\overline{\mathcal{R}_{2,1}(\boldsymbol{x}^\prime)} = \sum\limits_{l=1}^m (\overline{x}_{l}^2 + \overline{x}_{m+1}^2) |x_{l}^2 - x_{m+1}^2|^2,
    \\
    \mathcal{R}_{1,3}(\boldsymbol{x}^\prime)
        &=\overline{\mathcal{R}_{3,1}(\boldsymbol{x}^\prime)} =  \sum\limits_{l=1}^m \overline{x}_{l}^2 \overline{x}_{m+1}^2 |x_{l}^2 - x_{m+1}^2|^2, \
  %  \mathcal{R}_{2,1}(\boldsymbol{x}^\prime, \boldsymbol{y}^\prime)
   %     =  \sum\limits_{l=1}^m (x_{l}^2 + x_{m+1}^2) |x_{l}^2 - x_{m+1}^2|^2,
    \\
    \mathcal{R}_{2,2}(\boldsymbol{x}^\prime)
        &=  \sum\limits_{l=1}^m |x_{l}^4 - x_{m+1}^4|^2, \\
    \mathcal{R}_{2,3}(\boldsymbol{x}^\prime)
        &=\overline{\mathcal{R}_{3,2}(\boldsymbol{x}^\prime)} =  \sum\limits_{l=1}^m \overline{x}_{l}^2 \overline{x}_{m+1}^2 (x_{l}^2 + x_{m+1}^2) |x_{l}^2 - x_{m+1}^2|^2,
    \\
%    \mathcal{R}_{3,1}(\boldsymbol{x}^\prime, \boldsymbol{y}^\prime)
%        &=  \sum\limits_{l=1}^m x_{l}^2 x_{m+1}^2 |x_{l}^2 - x_{m+1}^2|^2, \
 %   \mathcal{R}_{3,2}(\boldsymbol{x}^\prime, \boldsymbol{y}^\prime)
 %       &=  \sum\limits_{l=1}^m x_{l}^2 x_{m+1}^2 (\overline{x}_{l}^2 + \overline{x}_{m+1}^2) |x_{l}^2 - x_{m+1}^2|^2,
 %   \\
    \mathcal{R}_{3,3}(\boldsymbol{x}^\prime)
        &=  \sum\limits_{l=1}^m |x_{l}|^4 |x_{m+1}|^4 |x_{l}^2 - x_{m+1}^2|^2,
\end{align*}
as well as the new denominator terms
\[
    \sum\limits_{\boldsymbol k\in {[m]\choose 2}} \big| (x_{k_1}^2 - x_{k_2}^2)(x_{k_1}^2 - x_{m+1}^2)(x_{k_2}^2 - x_{m+1}^2)\big|^2
\]
and applying Equation~\eqref{eq:aiSolGen2}.

Alternatively, fitting the data set $(\boldsymbol{x}, \boldsymbol{y})$, one may remember the % matrix $B$ as the denominator $\mathcal{D}(\boldsymbol{x})$ above and the
formula $\boldsymbol{a} =A^+\boldsymbol{y} = BB^{*}A^{*}\boldsymbol{y}$ where the matrix $B$ is given by equation \eqref{eq:B} with
entries
\[
    B_{1,\boldsymbol{l}}
    =          \frac{ x_{l_2}^2 - x_{l_1}^2 }{\sqrt{\mathcal{D}(\boldsymbol{x})}},
    \
    B_{2,\boldsymbol{l}}
    =          \frac{ x_{l_1}^4 - x_{l_2}^4 }{\sqrt{\mathcal{D}(\boldsymbol{x})}},
    \
    B_{3,\boldsymbol{l}}
    =          \frac{ x_{l_1}^2 x_{l_2}^2 (x_{l_2}^2 - x_{l_1}^2)}{\sqrt{\mathcal{D}(\boldsymbol{x})}},
    \
    \boldsymbol l\in {[m]\choose 2}.
\]
To compute the new matrix $B^\prime$ for fitting the enlarged data set, one would compute
\[
    \mathcal{D}(\boldsymbol{x}^\prime)
        =   \mathcal{D}(\boldsymbol{x})+\sum\limits_{\boldsymbol k\in {[m]\choose 2}} \big| (x_{k_1}^2 - x_{k_2}^2)(x_{k_1}^2 - x_{m+1}^2)(x_{k_2}^2 - x_{m+1}^2)\big|^2
\]
for the denominator
as well as $3m$ additional entries $B_{i,{l,m+1}}$ for $i = 1, 2, 3$ and $l = 1,\ldots, m$.

% xxxxxxxxxxxxxxxxxxxxxxxxxxxxxxxxxxxxxxxxxxxxxxxxxxxxxxxxxxxxxxxxxxxxxxxxx
% xxxxxxxxxxxxxxxxxxxxxxxxxxxxxxxxxxxxxxxxxxxxxxxxxxxxxxxxxxxxxxxxxxxxxxxxx
% xxxxxxxxxxxxxxxxxxxxxxxxxxxxxxxxxxxxxxxxxxxxxxxxxxxxxxxxxxxxxxxxxxxxxxxxx

\bibliographystyle{amsplain}
\bibliography{HHSregression}

\providecommand{\bysame}{\leavevmode\hbox to3em{\hrulefill}\thinspace}
\providecommand{\MR}{\relax\ifhmode\unskip\space\fi MR }
\providecommand{\MRhref}[2]{%
  \href{http://www.ams.org/mathscinet-getitem?mr=#1}{#2}
}
\providecommand{\href}[2]{#2}
\begin{thebibliography}{10}

\bibitem{ChinesePaper}
Q.~Chang, C.~Deng, and M.S. Floater, \emph{An interpolatory view of polynomial
  least squares approximation}, J. Approx. Theory \textbf{252} (2020), 105360,
  11 pp.

\bibitem{Cormen}
T.H. Cormen, C.E. Leiserson, R.L. Rivest, and C.~Stein, \emph{Introduction to
  algorithms}, 4th ed., The MIT Press, Cambridge, 2022.

\bibitem{Czekansky}
J.~Czekansky and T.~Sauer, \emph{The multivariate {Horner} scheme revisited},
  BIT \textbf{55} (2015), 1043--1056.

\bibitem{Gauss}
C.F. Gauss, \emph{Theoria combinationis observationum erroribus minimis
  obnoxiae}, Commentationes Societatis Regiae Scientiarum Gottingensis
  Recentiores. Comm. Class. Math., Vol. 2, H. Dieterich, 1823.

\bibitem{Gergonne}
J.D. Gergonne, \emph{The application of the method of least squares to the
  interpolation of sequences}, Historia Mathematica \textbf{1} (1974),
  439--447.

\bibitem{concrete}
R.L. Graham, D.E. Knuth, and O.~Patashnik, \emph{Concrete mathematics: A
  foundation for computer science}, 2nd ed., Addison-Wesley Publishing Group,
  Amsterdam, 1994.

\bibitem{Higham}
N.J. Higham, \emph{Accuracy and stability of numerical algorithms}, 2nd ed.,
  SIAM, Philadelphia, 2002.

\bibitem{Legendre}
A.M. Legendre, \emph{Nouvelles m{\'e}thodes pour la d{\'e}termination des
  orbites des com{\`e}tes}, Nineteenth Century Collections Online (NCCO):
  Science, Technology, and Medicine: 1780-1925, F. Didot, 1805.

\bibitem{Macdonald}
I.G. Macdonald, \emph{Symmetric functions and {Hall} polynomials. {With}
  contributions by {A}.{V}. {Zelevinsky}}, reprint of the 1998 2nd ed., Oxford
  University Press, Oxford, 2015.

\bibitem{Pan}
V.Y. Pan, \emph{How bad are {Vandermonde} matrices?}, SIAM J. Matrix Anal.
  Appl. \textbf{37} (2016), 676--694.

\bibitem{ReedSimon}
M.~Reed and B.~Simon, \emph{Methods of modern mathematical physics {I}:
  {Functional} analysis}, rev. and enl. ed., Academic Press, New York, 1980.

\bibitem{Sagan}
B.E. Sagan, \emph{The symmetric group: Representations, combinatorial
  algorithms, and symmetric functions}, 2nd ed., Graduate Texts in Mathematics
  {\bf 203}, Springer-Verlag, New York, 2001.

\bibitem{Stanley}
R.P. Stanley, \emph{Enumerative combinatorics. {V}ol. 2. {With} an appendix by
  {Sergey Fomin}}, 2nd ed., Cambridge Studies in Advanced Mathematics~{\bf
  208}, Cambridge University Press, Cambridge, 2024.

\bibitem{mathematica}
{Wolfram Research}, \emph{Mathematica 8.0}, 2010.

\end{thebibliography}

\end{document}